\documentclass[12pt,reqno]{amsart}

\usepackage{amssymb}

\usepackage[all]{xy}

\numberwithin{equation}{section}

\newtheorem{theorem}{Theorem}[section]
\newtheorem{proposition}[theorem]{Proposition}
\newtheorem{lemma}[theorem]{Lemma}

\theoremstyle{definition}

\newcommand{\C}{\mathbb{C}}

\newcommand{\GL}{\mathrm{GL}}

\begin{document}

\baselineskip=15pt

\title[]{Induced Representations on Character Varieties of Surfaces}

\author[I. Biswas]{Indranil Biswas}

\address{Department of Mathematics, Shiv Nadar University, NH91, Tehsil
Dadri, Greater Noida, Uttar Pradesh 201314, India}

\email{indranil.biswas@snu.edu.in, indranil29@gmail.com}

\author[J. Hurtubise]{Jacques Hurtubise}

\address{Department of Mathematics, McGill University, Burnside
Hall, 805 Sherbrooke St. W., Montreal, Que. H3A 2K6, Canada}

\email{jacques.hurtubise@mcgill.ca}

\author[L. C. Jeffrey]{Lisa C. Jeffrey}

\address{Department of Mathematics,
University of Toronto, Toronto, Ontario, Canada}

\email{jeffrey@math.toronto.edu}

\author[S. Lawton]{Sean Lawton}

\address{Department of Mathematical Sciences, George Mason University, 4400
University Drive, Fairfax, Virginia 22030, USA}

\email{slawton3@gmu.edu}



\date{\today}

\begin{abstract}
We show that a Frobenius reciprocity map on character varieties of surfaces is a Poisson embedding.
\end{abstract}

\maketitle

\section{Introduction}

Let $X,\, Y$ be closed oriented surfaces and $f \,:\, Y\,\longrightarrow\, X$ a degree $d$ branched
covering. Given a flat vector
bundle $(V,\, \nabla)$ of rank $r$ on $Y$, by taking the direct image (pushforward) we get a flat bundle 
$(f_*V,\, f_*\nabla)$ of rank $dr$ on $X\setminus B$ by Grauert's Theorem, where $B$ is the branch locus of $f$.

Denote $G_m\,=\,\GL(m,\C)$. The isomorphism class of the bundle $(V, \,\nabla)$ is determined by the
$G_r$--conjugation orbit of its holonomy $\rho\,:\,\pi_1(Y)\,\longrightarrow\, G_r$. Denote this orbit
(conjugacy class) by $[\rho]$. Consequently, the moduli space of such bundles is identified with the (affine)
Geometric Invariant Theory (GIT) quotient $\mathcal{M}_r(Y)\,:=\,\mathrm{Hom}(\pi_1(Y),\,G_r)/\!\!/G_r$.
We note that the {\it affine} GIT quotient is the affine variety whose coordinate ring is the ring of
$\mathrm{Inn}(G_r)$-invariant polynomials on $\mathrm{Hom}(\pi_1(Y),\,G_r)$. The above GIT quotient
$\mathrm{Hom}(\pi_1(Y),\,G_r)/\!\!/G_r$ is also known as a {\it character variety} or more precisely
the $G_r$--character variety of $\pi_1(Y)$. We likewise consider the character variety
$\mathcal{M}_{dr}(X-B)\,:=\,\mathrm{Hom}(\pi_1(X-B),\,G_{dr})/\!\!/G_{dr}$.

By \cite{Go, La, BJ}, we know $\mathcal{M}_r(Y)$ is (holomorphic) symplectic and $\mathcal{M}_{dr}(X-B)$ is 
(holomorphic) Poisson on their smooth loci. By sending the holonomy of $(V,\, \nabla)$ to the holonomy of 
$(f_*V,\, f_*\nabla)$, we get a natural map $\mathsf{Frob} \,:\, \mathcal{M}_r(Y)
\,\longrightarrow\, \mathcal{M}_{dr}(X-B)$ which we call the {\it Frobenius reciprocity map}.

\begin{theorem}\label{thm:main}
 $\mathsf{Frob}$ is a Poisson embedding.
\end{theorem}

Let $Y_0\,=\,Y-f^{-1}(B)$ and $\mathcal{M}_r(Y_0)\,:=\,\mathrm{Hom}(\pi_1(Y_0),\,G_r)/\!\!/G_r$. Again,
$\mathcal{M}_r(Y_0)$ is (holomorphic) Poisson on its smooth locus. Since $Y_0\,\subset\, Y$ there is a
natural map $\iota\,:\,\pi_1(Y_0)\,\longrightarrow\, \pi_1(Y)$. The map $\iota$ is surjective (and hence has a
right inverse) because any loop in $Y$ can be homotoped to avoid $f^{-1}(B)$ since $\dim f^{-1}(B)\,=\,0$.

Let $\mathsf{Res}\,:\,\mathcal{M}_r(Y)\,\longrightarrow\, \mathcal{M}_r(Y_0)$ be given by $\mathsf{Res}([h])
\,=\,[h\circ \iota]$, where $h\,:\,\pi_1(Y)\,\longrightarrow\, G_r$. If
$\mathsf{Res}([h_1])\,=\,\mathsf{Res}([h_1])$, then 
$[h_1\circ\iota]\,=\,[h_2\circ\iota]$ and so $[h_1]=[h_2]$ since $\iota$ has a right inverse. Therefore, 
$\mathsf{Res}$ is an embedding. The inclusion map $Y_0\,\subset\, Y$ preserves transversality of based loops 
and thus by \cite[Theorem A]{BHJL} the ``restriction" map $\mathsf{Res}$ is Poisson.

Since $f\,:\,Y_0\,\longrightarrow\, X-B$ is an unramified covering map, $\pi_1(Y_0)$ is isomorphic to a
subgroup of $\pi_1(X-B)$. Therefore, there
is a map $\mathsf{Ind}\,:\,\mathcal{M}_r(Y_0)\,\longrightarrow\, \mathcal{M}_{dr}(X-B)$ given by 
$\mathsf{Ind}([h])\,=\,[\mathsf{ind}(h)]$, where $\mathsf{ind}(h)\,:\,\pi_1(X-B)\,\longrightarrow\, G_{dr}$
is the representation induced by $h\,:\,\pi_1(Y_0)\,\longrightarrow\, G_r$. Since $f\,:\,Y_0\,\longrightarrow\,
X-B$ is a $d$--sheeted cover, the index $[\pi_1(X-B):\,\pi_1(Y_0)]\,=\,d$ is finite and 
$\pi_1(X-B)/\pi_1(Y_0)\,=\, \{\gamma_1\pi_1(Y_0),\, \cdots,\,\gamma_d\pi_1(Y_0)\}$. And
so $\mathsf{ind}(h)$ is defined 
by the natural action of $\pi_1(X-B)$ on $\oplus_{i=1}^d\gamma_i\C^r\,\cong\,\C^{dr}$ by $\gamma\cdot 
\sum_{i=1}^d\gamma_iv_i\,:=\,\sum_{i=1}^d\gamma_{j_i}\pi(\delta_i)v_i$, where 
$\gamma\gamma_i\,=\,\gamma_{j_i}\delta_i$ with $\delta_i\,\in\,\pi_1(Y_0)$. Thus, $\mathsf{ind}(h)$ takes
values in $\mathsf{Sym}_d\times G_r$ inside $G_{dr}$ where $\mathsf{Sym}_d$ is the symmetric group on $d$ symbols 
(realized by permutation matrices).

By construction of the direct image sheaf, we have $\mathsf{Frob}\,=\,\mathsf{Ind}\circ \mathsf{Res}$ (see 
Section \ref{sec:directimage} for details). Consequently, $\mathsf{Frob}$ is algebraic. By the Mackey 
Formula \cite{Mackey} applied to this context, we see that $\mathsf{Ind}$ is injective. Therefore, 
$\mathsf{Frob}$ is an embedding. In Section \ref{sec3}, we show that $\mathsf{Ind}$ is Poisson. Having 
both $\mathsf{Res}$ and $\mathsf{Ind}$ injective and Poisson, we have outlined the proof of Theorem 
\ref{thm:main}.

\section{Direct image of a connection}\label{sec:directimage}

Let $X$ and $Y$ be compact connected $C^\infty$ oriented surfaces, and let
\begin{equation}\label{e1}
f\ :\ Y\ \longrightarrow\ X
\end{equation}
be an orientation preserving ramified covering map. This means the following:
\begin{enumerate}
\item $f$ is a $C^\infty$ surjective map,

\item there is finite subset $B_0\, \subset\, X$ such that the restriction
$$
f\big\vert_{f^{-1}(X\setminus B_0)}\ :\ f^{-1}(X\setminus B_0)\ \longrightarrow\ X\setminus B_0
$$
is an unramified covering map, and

\item for each point $y\, \in\, f^{-1}(B_0)$, there is a positive integer $n_y$ such that
the restriction of $f$ to a small open neighborhood of $y$ is conjugate to the map
$z\, \longmapsto\, z^{n_y}$ defined around $0\,\in\, {\mathbb C}$. 
\end{enumerate}
We note that $B_0$ is allowed to be the empty set.

Let $B\, \subset \, X$ be a subset containing $B_0$. Let
\begin{equation}\label{e2}
R \ :=\ f^{-1}(B) \ \subset\ Y
\end{equation}
be the inverse image of $B$. The open subsets $X\setminus B\, \subset\, X$ and
$Y\setminus R\, \subset\, Y$ will be denoted by $X_0$ and $Y_0$ respectively. So
the restriction of $f$
\begin{equation}\label{e3}
\phi\ :=\ f\big\vert_{Y_0} \ :\ Y_0\ \longrightarrow\ X_0
\end{equation}
is an unramified covering map (recall that $B_0\, \subset\, B$). Fix a base point $y_0\, \in\, Y_0$. Denote
the fundamental
groups $\pi_1(X_0,\, \phi(y_0))$ and $\pi_1(Y_0,\, y_0)$ by $\Gamma$ and $\Delta$ respectively.
Using the homomorphism
$$
\phi_* \ :\ \Delta\ =\ \pi_1(Y_0,\, y_0) \ \longrightarrow \ \pi_1(X_0,\, \phi(y_0))\ =\ \Gamma
$$
induced by $\phi$ in \eqref{e3}, we will consider $\Delta$ as a finite index subgroup of $\Gamma$.

Denote the degree of the map $f$ in \eqref{e1} by $d$. Let $V$ be a $C^\infty$ complex vector bundle of
rank $r$ on $Y_0$ equipped with a flat connection
\begin{equation}\label{e4}
\nabla\ :\ V \ \longrightarrow\ V\otimes T^*Y_0 .
\end{equation}
Consider the $C^\infty$ complex vector bundle on $X_0$ of rank
$dr$ given by the direct image $\phi_* V$. Since $\phi$ is an unramified covering map, taking the direct
image --- using $\phi$ --- of the connection $\nabla$ in \eqref{e4} we have
\begin{equation}\label{e31}
\phi_*\nabla \ :\ \phi_* V \ \longrightarrow\ \phi_*(V\otimes T^*Y_0)\ =\ (\phi_* V)\otimes T^*X_0.
\end{equation}
{}From the Leibniz identity satisfied by $\nabla$ it follows that $\phi_*\nabla$ also satisfies
the Leibniz identity. Hence $\phi_*\nabla$ is a connection on $\phi_* V$. It is straightforward to
check that the connection $\phi_*\nabla$ on $\phi_* V$ is flat (in fact, in general
the curvature of $\phi_*\nabla$ is given by the direct image of the curvature of $\nabla$).

\begin{proposition}\label{prop-induced}
The monodromy representation of $\Gamma$ --- for the flat connection $\phi_*\nabla$ --- coincides with the
induced representation associated to the monodromy representation of the subgroup $\Delta\, \subset\,
\Gamma$ corresponding to the flat connection $\nabla$. 
\end{proposition}

\begin{proof}
To explain this, let
$$
\rho\ :\ \Delta \ \longrightarrow\ \text{GL}(V_{y_0})
$$
be the monodromy representation for the flat connection $\nabla$. Using $\rho$, consider $V_{y_0}$
as a left $\Delta$--module. Let $W$ denote the space of all maps $\beta\, :\, \Gamma\,
\longrightarrow\, V_{y_0}$ satisfying the following condition:
\begin{equation}\label{e32}
\beta(\gamma\delta)\ =\ \rho(\delta^{-1})(\beta(\gamma))
\end{equation}
for all $\gamma\, \in\, \Gamma$ and $\delta\, \in\, \Delta$. Note that $W$ is a complex vector
space of dimension $\text{degree}(f)\cdot \dim V_{y_0}\,=\, dr$. The left-translation action of $\Gamma$ on
itself produces an action of $\Gamma$ on $W$. More explicitly, the action of any element
$\gamma\, \in\, \Gamma$ on $W$ sends any $\beta\, :\, \Gamma\, \longrightarrow\, V_{y_0}$ to the
map $\gamma' \, \longmapsto\, \beta(\gamma\gamma')$, $\gamma'\, \in\, \Gamma$. The vector bundle
with flat connection on $X_0$ given by the $\Gamma$--module $W$ is canonically identified with
$\phi_* V$ equipped with the flat connection $\phi_*\nabla$. We will briefly recall the construction of this canonical isomorphism.

Let
\begin{equation}\label{a1}
({\mathcal W},\, \nabla^W)
\end{equation}
be the vector bundle on $X_0$ with flat connection given
by the $\Gamma$--module $W$. We recall that
$${\mathcal W}\ =\ (\widetilde{X}_0\times W)/\Gamma ,$$
where
\begin{equation}\label{e6}
\varpi\ :\ \widetilde{X}_0\ \longrightarrow\ X_0
\end{equation}
is the universal cover of $X_0$ for the base point $\phi(y_0)$, and
$\pi_1(X_0,\, \phi(y_0))$ acts twisted diagonally on $\widetilde{X}_0\times W$ using the deck transformations
of the covering $\widetilde{X}_0$ and the $\Gamma$--module structure of $W$. More precisely,
two points $(x_1,\, w_1)$ and $(x_2,\, w_2)$ of $\widetilde{X}_0\times W$ lie over the same point
of $\mathcal W$ if and only if there is $\gamma\, \in\, \Gamma$ such that $x_2\,=\, x_1\gamma$ and
$w_2\,=\, \gamma^{-1}\cdot w_1$. In particular, there is a canonical isomorphism of the fiber
${\mathcal W}_{\phi(y_0)}$ with $W$ that sends any $w\, \in\, W$ to the equivalence class
of $(\widetilde{\phi(y_0)},\, w)$, where $\widetilde{\phi(y_0)}\, \in\, \widetilde{X}_0$ is the
base point of $\widetilde{X}_0$ over $\phi(y_0)$. The trivial connection
on the trivial vector bundle $\widetilde{X}_0\times W\, \longrightarrow\, \widetilde{X}_0$ is preserved
by the action of $\Gamma$ and hence it descends to a flat connection on the vector bundle
$${\mathcal W}\ :=\ (\widetilde{X}_0\times W)/\Gamma\ \longrightarrow\ \widetilde{X}_0/\Gamma\ =\ X_0;$$
this descended flat connection on ${\mathcal W}$ is denoted by $\nabla^W$.

The fiber $(\phi_* V)_{\phi(y_0)}$ of $\phi_*V$ over $\phi(y_0)$ is identified with
$\bigoplus_{z\in \phi^{-1}(\phi(y_0))} V_z$. We will now show that there is a canonical isomorphism
\begin{equation}\label{e5}
I \ :\ {\mathcal W}_{\phi(y_0)}\ =\ W \ \longrightarrow\ (\phi_* V)_{\phi(y_0)}\
=\ \bigoplus_{z\in \phi^{-1}(\phi(y_0))} V_z
\end{equation}
of fibers over the base point $\phi(y_0)\, \in\, X_0$ (it was noted above that
${\mathcal W}_{\phi(y_0)}\, =\, W$).

Let ${\mathcal F}$ denote the space of all maps from
$\varpi^{-1}(\phi(y_0))$ (see \eqref{e6} for $\varpi$) to $V$. The action of $\Gamma$ on
$\varpi^{-1}(\phi(y_0))$ produces an action of $\Gamma$ on ${\mathcal F}$.
The universal cover of $Y_0$ for the point $y_0$ coincides with the universal cover $\widetilde{X}_0$
of $X_0$ for the point $\phi(y_0)$; in fact, we have $Y_0\,=\, \widetilde{X}_0/\Delta$. Using this it follows
that $W\,=\, {\mathcal F}/\Delta$. In addition, it can be shown that $\bigoplus_{z\in \phi^{-1}(\phi(y_0))} V_z\,=\,
{\mathcal F}/\Delta$. Indeed, this follows from the fact the pullback of the flat
vector bundle $(V,\, \nabla)$ to the universal
cover $\widetilde{X}_0$ of $Y_0$ is canonically identified with the trivial vector bundle
$\widetilde{X}_0\times V\, \longrightarrow\, \widetilde{X}_0$ equipped with the trivial connection.
Combining the above two isomorphisms $W\,=\, {\mathcal F}/\Delta$ and $\bigoplus_{z\in \phi^{-1}(\phi(y_0))} V_z
\,=\, {\mathcal F}/\Delta$ we get an isomorphism $W\, \stackrel{\sim}{\longrightarrow}\,
\bigoplus_{z\in \phi^{-1}(\phi(y_0))} V_z$; this isomorphism is denoted by $I$ in \eqref{e5}.

Using $I$ in \eqref{e5}, we will construct an isomorphism of vector bundles
\begin{equation}\label{e7}
{\mathcal I}\ :\ {\mathcal W}\ \longrightarrow\ \phi_* V.
\end{equation}

Take any point $x\, \in\, X_0$. Fix a homotopy class of path $\gamma$ connecting $x$ to the base point
$\phi(y_0)$. The parallel translation, along $\gamma$, for the flat connection $\nabla^W$ on $\mathcal W$
produces an isomorphism
\begin{equation}\label{e8}
\Phi_x\ :\ {\mathcal W}_x\ \longrightarrow\ {\mathcal W}_{\phi(y_0)}.
\end{equation}
For each point $z\, \in\, \phi^{-1}(\phi(y_0))$, let $\gamma_z$ be the unique lift of $\gamma$ passing
through $z$. Denote by $\alpha(z)$ the point of $\gamma_z$ over $x$; so $\gamma_z$ connects $\alpha(z)$
with $z$. Consider parallel translation, along $\gamma_z$, for the connection $\nabla$ on $V$
in \eqref{e4}. This produces an isomorphism
$$
\widetilde{\Phi}_z\ :\ V_{\alpha(z)}\ \longrightarrow\ V_z.
$$
Combining these isomorphisms $\widetilde{\Phi}_z$ we have an isomorphism
\begin{equation}\label{e9}
\bigoplus_{z \in\ \phi^{-1}(\phi(y_0))} \widetilde{\Phi}^{-1}_z\ :\ 
\bigoplus_{z \in\ \phi^{-1}(\phi(y_0))} V_z\ \longrightarrow\ 
\bigoplus_{z \in\ \phi^{-1}(\phi(y_0))}
V_{\alpha(z)}.
\end{equation}
Now consider the homomorphism
\begin{equation}\label{e10}
{\mathcal I}'_x\ :=\ \left(\bigoplus_{z \in\ \phi^{-1}(\phi(y_0))} \widetilde{\Phi}^{-1}_z\right)
\circ I\circ \Phi_x
\ :\ {\mathcal W}_x\ \longrightarrow\ \bigoplus_{z \in\ \phi^{-1}(\phi(y_0))}
V_{\alpha(z)},
\end{equation}
where $\Phi_x$, $I$ and $\bigoplus_{z \in\ \phi^{-1}(\phi(y_0))} \widetilde{\Phi}^{-1}_z$ are constructed
in \eqref{e8}, \eqref{e5} and \eqref{e9} respectively. Note that ${\mathcal I}'_x$ is an isomorphism because
$\Phi_x$, $I$ and $\bigoplus_{z \in\ \phi^{-1}(\phi(y_0))} \widetilde{\Phi}^{-1}_z$ are all isomorphisms
(as noted above). Since the above map $z\, \longmapsto\, \alpha(z)$, $z\, \in\,
\phi^{-1}(\phi(y_0))$, is a bijection between $\phi^{-1}(\phi(y_0))$ and $\phi^{-1}(x)$, we have
$$
\bigoplus_{z \in\ \phi^{-1}(\phi(y_0))}V_{\alpha(z)}\ =\ \bigoplus_{z'\in\ \phi^{-1}(x)}V_{z'}
\ =\ (\phi_* V)_x.
$$
Consequently, the isomorphism ${\mathcal I}'_x$ in \eqref{e10} produces an isomorphism
\begin{equation}\label{e11}
{\mathcal I}_x\ :=\ \left(\bigoplus_{z \in\ \phi^{-1}(\phi(y_0))}\widetilde{\Phi}^{-1}_z\right) \circ I\circ \Phi_x
\ :\ {\mathcal W}_x\ \longrightarrow\ (\phi_* V)_x.
\end{equation}
It is straightforward to check that the isomorphism ${\mathcal I}_x$ in \eqref{e11} does not depend on
the homotopy class of $\gamma$. Consequently, we get an isomorphism $\mathcal I$ as in \eqref{e7} whose
restriction over $x$ coincides with ${\mathcal I}_x$ for every $x\, \in\, X_0$.

From the construction of the isomorphism ${\mathcal I}$ in \eqref{e7} it follows immediately that
${\mathcal I}$ takes the flat connection $\nabla^W$ on ${\mathcal W}$ to the flat connection 
$\phi_*\nabla$ on $\phi_* V$.
\end{proof}

\section{Map between character varieties}\label{sec3}

Let $(V_1,\, \nabla_1)$ and $(V_2,\, \nabla_2)$ be two flat vector bundles over $Y_0$ of rank $r$
such that for every $z\, \in\, R\,:=\, Y\setminus Y_0$ (see \eqref{e2} for $R$), the conjugacy
class of the local monodromy of $\nabla_1$ around $z$ coincides with the conjugacy
class of the local monodromy of $\nabla_2$ around $z$. Let $({\mathcal W}_1,\, \nabla^{W_1})$
(respectively, $({\mathcal W}_2,\, \nabla^{W_2})$) be the flat vector bundle on $X_0$ constructed from
$(V_1,\, \nabla_1)$ (respectively, $(V_2,\, \nabla_2)$) (see \eqref{a1}). From the constructions
of $({\mathcal W}_1,\, \nabla^{W_1})$ and $({\mathcal W}_2,\, \nabla^{W_2})$ it can be deduced
that for every $x\, \in\, B\,=\, X\setminus X_0$, the conjugacy
class of the local monodromy of $\nabla^{W_1}$ around $x$ coincides with the conjugacy
class of the local monodromy of $\nabla^{W_2}$ around $x$. Indeed, it is straightforward to
check that for every $x\, \in\, B\,=\, X\setminus X_0$, the conjugacy
class of the local monodromy --- around $x$ --- of the connection $\phi_*\nabla_1$ on $\phi_*V_1$
coincides with the conjugacy class of the local monodromy of $\phi_*\nabla_2$ around $x$. Now 
the earlier observation that the isomorphism $\mathcal I$ in \eqref{e7} is connection preserving implies
that for every $x\, \in\, B\,=\, X\setminus X_0$, the conjugacy
class of the local monodromy of $\nabla^{W_1}$ around $x$ coincides with the conjugacy
class of the local monodromy of $\nabla^{W_2}$ around $x$.

Let $\text{Hom}(\Delta,\, \text{GL}(r,{\mathbb C}))^{\rm ir}\, \subset\, \text{Hom}(\Delta,\, \text{GL}
(r,{\mathbb C}))$ be the space of irreducible representations of $\Delta$ in $\text{GL}(r,{\mathbb C})$.
The conjugation action of $\text{GL}(r,{\mathbb C})$ on itself produces an action of $\text{GL}
(r,{\mathbb C})$ on $\text{Hom}(\Delta,\, \text{GL}(r,{\mathbb C}))^{\rm ir}$. The corresponding quotient 
\begin{equation}\label{e12}
{\mathcal C}_r(Y_0)\ :=\ \text{Hom}(\Delta,\, \text{GL}(r,{\mathbb C}))^{\rm ir}/\text{GL}(r,{\mathbb C})
\end{equation}
is a complex manifold equipped with a holomorphic Poisson structure \cite{La}, \cite{BJ}, \cite{Go}.
Similarly, $\text{Hom}(\Gamma,\, \text{GL}(rd,{\mathbb C}))^{\rm ir}\, \subset\, \text{Hom}(\Gamma,\, \text{GL}
(rd,{\mathbb C}))$ is the space of irreducible representations of $\Gamma$ in $\text{GL}(rd,{\mathbb C})$.
As before,
\begin{equation}\label{e13}
{\mathcal C}_{rd}(X_0)\ :=\ \text{Hom}(\Gamma,\, \text{GL}(rd,{\mathbb C}))^{\rm ir}/\text{GL}(rd,{\mathbb C}) 
\end{equation}
is the quotient complex manifold equipped with a holomorphic Poisson structure. Sending
any $(V,\, \nabla)\, \in\, {\mathcal C}_r(Y_0)$ to its direct image $(\phi_* V,\, \phi_*\nabla)\, \in\,
{\mathcal C}_{rd}(X_0)$ (see \eqref{e12} and \eqref{e13}) we obtain a map
\begin{equation}\label{e14}
\Psi\ :\ {\mathcal C}_r(Y_0)\ \longrightarrow\ {\mathcal C}_{rd}(X_0).
\end{equation}

\begin{theorem}\label{thm1}
The map $\Psi$ in \eqref{e14} is compatible with the Poisson structures on ${\mathcal C}_r(Y_0)$
and ${\mathcal C}_{rd}(X_0)$.
\end{theorem}

\begin{proof}
The holomorphic tangent (respectively, cotangent) bundle of ${\mathcal C}_r(Y_0)$ will be
denoted by $T{\mathcal C}_r(Y_0)$ (respectively, $T^*{\mathcal C}_r(Y_0)$). The Poisson structure
on ${\mathcal C}_r(Y_0)$ is given by a holomorphic homomorphism
\begin{equation}\label{e15}
P\ :\ T^*{\mathcal C}_r(Y_0) \ \longrightarrow\ T{\mathcal C}_r(Y_0)
\end{equation}
which is skew-symmetric and the Poisson bracket defined by $P$ satisfies the Jacobi identity. Similarly,
\begin{equation}\label{e16}
Q\ :\ T^*{\mathcal C}_{rd}(X_0) \ \longrightarrow\ T{\mathcal C}_{rd}(X_0)
\end{equation}
is the Poisson structure on ${\mathcal C}_{rd}(X_0)$. The map $\Psi$ in \eqref{e14} is said to be
compatible with the Poisson structures on ${\mathcal C}_r(Y_0)$ and ${\mathcal C}_{rd}(X_0)$ if the
following diagram of homomorphisms of vector bundles over ${\mathcal C}_r(Y_0)$ is commutative:
\begin{equation}\label{e17}
\begin{matrix}
T^*{\mathcal C}_r(Y_0) & \xrightarrow{\,\,\, P\,\,\,} & T{\mathcal C}_r(Y_0)\\
\, \,\,\,\,\, \Big\uparrow (d\Psi)^* && \,\,\,\, \Big\downarrow d\Psi\\
\Psi^*T^*{\mathcal C}_{rd}(X_0) & \xrightarrow{\,\,\, \Psi^*Q\,\,\,} & \Psi^* T{\mathcal C}_{rd}(X_0)
\end{matrix}
\end{equation}
where $P$ (respectively, $Q$) is constructed in \eqref{e15} (respectively, \eqref{e16})
and $d\Psi$ is the differential of the map $\Psi$ in \eqref{e14} while $(d\Psi)^*$ is its dual.

To prove that the diagram in \eqref{e17} is commutative, we first briefly recall the construction of
the homomorphisms $P$ and $Q$ (more details can be found in \cite{BJ}).

Take any flat vector bundle $\alpha\, =\, (V,\, \nabla^V)\, \in\, {\mathcal C}_r(Y_0)$ on $Y_0$ of rank $r$.
The locally constant sheaf on $Y_0$ given by the locally defined flat
sections of $V$ for the connection $\nabla$ will be denoted by $\underline{V}$.
The flat connection on $\text{End}(V)\,=\, V\otimes V^*$ induced by $\nabla^V$ will be denoted by $\nabla$.
Let $\underline{\rm End}(V)$ be the locally constant sheaf on $Y_0$ given by the locally defined flat
sections of $\text{End}(V)$ for the induced connection $\nabla$. Note that we have
\begin{equation}\label{a2}
\underline{\rm End}(V)\ = \ \text{End}(\underline{V}).
\end{equation}
The tangent space of ${\mathcal C}_r(Y_0)$ at the point $\alpha\, =\, (V,\, \nabla^V)$ has the
following description:
\begin{equation}\label{a4}
T_\alpha {\mathcal C}_r(Y_0)\ =\ H^1(Y_0,\, \underline{\rm End}(V)).
\end{equation}
Note that $\underline{\rm End}(V)\,=\, \underline{\rm End}(V)^*$; the isomorphism is given by the
trace map. Hence, using Poincar\'e duality,
\begin{equation}\label{a5}
T^*_\alpha {\mathcal C}_r(Y_0)\ =\ H^1_c(Y_0,\, \underline{\rm End}(V)),
\end{equation}
where $H^i_c$ denotes $i$--th cohomology with compact support. In terms of the isomorphisms
in \eqref{a4} and \eqref{a5}, the homomorphism $P$ in \eqref{e15} coincides with the natural map
\begin{equation}\label{e17b}
T^*_\alpha {\mathcal C}_r(Y_0)\ =\ H^1_c(Y_0,\, \underline{\rm End}(V))\ \longrightarrow\
H^1(Y_0,\, \underline{\rm End}(V))\ =\ T_\alpha {\mathcal C}_r(Y_0).
\end{equation}

Consider $$\beta \,:=\, \Psi(\alpha) \, =\, (\phi_*V,\, \phi_*\nabla^V)\, =:\,
(W,\, \phi_*\nabla^W) \, \in\, {\mathcal C}_{rd}(X_0),$$
where $\Psi$ and $\phi$ are the maps in \eqref{e14} and \eqref{e3} respectively.
The locally constant sheaf on $X_0$ given by the locally defined flat
sections of $W$ for the connection $\nabla^W$ will be denoted by $\underline{W}$.
Let $\widehat{\nabla}$
be the flat connection on $\text{End}(W)\,=\, W\otimes W^*$ induced by $\nabla^W$. Denote by
$\underline{\rm End}(W)$ the locally constant sheaf on $X_0$ given by the locally defined flat
sections of $\text{End}(W)$ for the induced connection $\widehat{\nabla}$; so we have
\begin{equation}\label{a3}
\underline{\rm End}(W)\ = \ \text{End}(\underline{W}).
\end{equation}
As before, we have
\begin{equation}\label{a6}
T_\beta {\mathcal C}_{rd}(X_0)\ =\ H^1(X_0,\, \underline{\rm End}(W)),\ \
T^* _\beta {\mathcal C}_{rd}(X_0)\ =\ H^1_c(X_0,\, \underline{\rm End}(W)).
\end{equation}

Next we will describe that homomorphism $d\Psi$ in \eqref{e17}.

For finite dimensional vector spaces $F_1,\, F_2,\, \ldots,\, F_n$, the direct sum
$\bigoplus_{j=1}^n \text{End}(F_j)$ is canonically a direct summand of
$\text{End}\left(\bigoplus_{j=1}^nF_j\right)$. The canonical complement of it is
$\bigoplus_{i,j=1, i\not= j}^n \text{Hom}(F_i,\, F_j)$. Note that the stalk of the locally constant
sheaf $\phi_* \underline{V}$ (see \eqref{a2} and \eqref{e3} for
$\underline{V}$ and $\phi$ respectively) at any $x\, \in\, X_0$ is the direct
sum of the stalks of $\underline{V}$ at the points of $\phi^{-1}(x)$. Using this, from the above observation
that $\bigoplus_{j=1}^n \text{End}(F_j)$ is canonically a direct summand of
$\text{End}\left(\bigoplus_{j=1}^nF_j\right)$ it follows that the locally constant sheaf
$\phi_* \text{End}(\underline{V})$ on $X_0$ is canonically a direct summand of
$\text{End}(\phi_*\underline{V})$. In particular, we have a canonical injective homomorphism
\begin{equation}\label{e18}
H^1(X_0,\, \phi_* \text{End}(\underline{V})) \ \hookrightarrow\
H^1(X_0,\, \text{End}(\phi_*\underline{V})).
\end{equation}
Since $\phi$ is an unramified covering map,
$$
H^1(X_0,\, \phi_* \text{End}(\underline{V}))\ =\ H^1(Y_0,\, \text{End}(\underline{V})).
$$
Combining this with \eqref{e18}, we get an injective homomorphism
$$
H^1(Y_0,\, \text{End}(\underline{V}))\ \hookrightarrow\ H^1(X_0,\, \text{End}(\phi_*\underline{V})).
$$
In view of \eqref{a4} and \eqref{a6}, this injective homomorphism produces an injective
homomorphism
\begin{equation}\label{e19}
T_\alpha {\mathcal C}_r(Y_0)\ =\ H^1(Y_0,\, \text{End}(\underline{V}))\ \longrightarrow
\ H^1(X_0,\, \text{End}(\phi_*\underline{V}))\ =\ T_\beta {\mathcal C}_{rd}(X_0).
\end{equation}
The homomorphism $d\Psi$ in \eqref{e17} at the point $\alpha$ coincides with the injective homomorphism in \eqref{e19}.

Since $\phi_* \text{End}(\underline{V})$ is canonically a direct summand of $\text{End}(\phi_*\underline{V})$,
there is a natural projection $\text{End}(\phi_*\underline{V})\, \longrightarrow\, \phi_* \text{End}(
\underline{V})$. Note that using the fact that $\phi$ is an unramified covering map --- which implies that
$\phi$ is a proper map --- we have a homomorphism
\begin{equation}\label{e20-1}
H^1_c(X_0,\, \text{End}(\phi_*\underline{V}))\ \longrightarrow\ H^1_c(X_0,\,
\phi_* \text{End}(\underline{V}))\ =\ H^1_c(Y_0,\, \text{End}(\underline{V})).
\end{equation}
Note that the homomorphism in \eqref{e20-1} is surjective because $\phi_* \text{End}(\underline{V})$ is
a direct summand of $\text{End}(\phi_*\underline{V})$.
Now incorporating \eqref{a5} and \eqref{a6} into \eqref{e20-1}, we get a homomorphism
\begin{equation}\label{e20}
T^*_\beta {\mathcal C}_{rd}(X_0)\ =\ H^1_c(X_0,\, \text{End}(\phi_*\underline{V}))
\ \longrightarrow\ H^1_c(Y_0,\, \text{End}(\underline{V}))\ =\ T^*_\alpha {\mathcal C}_r(Y_0);
\end{equation}
this homomorphism is surjective because the homomorphism in \eqref{e20-1} is so.

The homomorphism $(d\Psi)^*$ in \eqref{e17} at $\beta$ coincides with the surjective homomorphism in \eqref{e20}.

Using the above constructions of the homomorphisms in \eqref{e17} it is straightforward to deduce
that the diagram in \eqref{e17} is commutative.
\end{proof}

\section{Pullback of flat connections}

Let $E$ be a $C^\infty$ vector bundle of rank $n$ on $X_0$ equipped with a flat connection $\nabla^E$.
So the pullback $(\phi^*E, \phi^*\nabla^E)$, by the map $\phi$ in \eqref{e3}, is a vector bundle on $Y_0$
with a flat connection. As before, define
\begin{equation}\label{e21}
{\mathcal C}_{n}(X_0)\ :=\ \text{Hom}(\Gamma,\, \text{GL}(n,{\mathbb C}))^{\rm ir}/\text{GL}(n,{\mathbb C}),
\end{equation}
\begin{equation}\label{e22}
{\mathcal C}_n(Y_0)\ :=\ \text{Hom}(\Delta,\, \text{GL}(n,{\mathbb C}))^{\rm ir}/\text{GL}(n,{\mathbb C}).
\end{equation}
Sending any $(E,\, \nabla^E)\, \in\, {\mathcal C}_{n}(X_0)$ to $(\phi^*E, \phi^*\nabla^E)$ we construct
a map
\begin{equation}\label{e23}
\Phi\ :\ {\mathcal C}_{n}(X_0)\ \longrightarrow\ {\mathcal C}_n(Y_0).
\end{equation}

\begin{proposition}\label{prop1}
The map $\Phi$ in \eqref{e23} is compatible with the Poisson structures on ${\mathcal C}_n(Y_0)$
and ${\mathcal C}_{n}(X_0)$.
\end{proposition}

\begin{proof}
We will describe the differential $d\Phi$ of the map in \eqref{e23}. Take
$(E,\, \nabla^E)\, \in\, {\mathcal C}_{n}(X_0)$. Denote by $\underline{E}$ the locally
constant sheaf on
$X_0$ given by the locally defined flat sections of $E$ for the connection $\nabla^E$.
The flat connection on $\text{End}(E)\,=\, E\otimes E^*$ induced by $\nabla^E$ will
be denoted by $\nabla$. Let $\underline{\rm End}(E)$ be the locally constant sheaf on $X_0$ given by the
locally defined flat sections of $\text{End}(E)$ for the induced connection $\nabla$. We have
$\underline{\rm End}(E)\, = \, \text{End}(\underline{E})$. The locally
constant sheaf on
$Y_0$ given by the locally defined flat sections of $\phi^*E$ for the flat
connection $\phi^*\nabla^E$ coincides with the pullback $\phi^*\underline{E}$.
The locally constant sheaf on $Y_0$, given by the locally defined flat
sections of $\text{End}(\phi^* E)\,=\, \phi^*\text{End}(E)$ for the
flat connection induced by $\phi^* \nabla$, coincides with $$\phi^*\underline{\rm End}(E)
\ = \ \phi^* \text{End}(\underline{E}) \ = \ \text{End}(\phi^*\underline{E}).$$

We have $T_{(E,\nabla^E)}{\mathcal C}_{n}(X_0)\,=\, H^1(X_0,\, \text{End}(\underline{E}))$ and
$$T_{(\phi^*E,\phi^*\nabla^E)}{\mathcal C}_{n}(Y_0)\,=\, H^1(Y_0,\, \text{End}(\phi^* \underline{E}))
\ =\ H^1(Y_0,\, \phi^*\text{End}(\underline{E})).
$$
using these two isomorphisms, the differential $d\Phi \,:\, T_{(E,\nabla^E)}{\mathcal C}_{n}(X_0)
\,\longrightarrow\, T_{(\phi^*E,\phi^*\nabla^E)}{\mathcal C}_{n}(Y_0)$ at the point $(E,\,
\nabla^E)\, \in\, {\mathcal C}_{n}(X_0)$ coincides with the natural pullback homomorphism
\begin{equation}\label{e24}
H^1(X_0,\, \text{End}(\underline{E}))\ \longrightarrow\ 
H^1(Y_0,\, \phi^*\text{End}(\underline{E})).
\end{equation}
The differential $d\Phi$ of the map in \eqref{e23} at the point $(E,\, \nabla^E)\, \in\,{\mathcal C}_{n}(X_0)$
coincides with the homomorphism in \eqref{e24}. Next, the dual homomorphism $(d\Phi)^*$ will be described.

First it will be shown that there is a natural surjective homomorphism
\begin{equation}\label{e25}
H\ :\ \phi_* \phi^* E \ \longrightarrow\ E
\end{equation}
that has a canonical splitting.
For any $x\, \in\, X_0$, the fiber of $(\phi_* \phi^* E)_x$ is the direct sum
$$\bigoplus_{y\in \phi^{-1}(x)} (\phi^* E)_y\ =\ \bigoplus_{y\in \phi^{-1}(x)} E_x.$$
So we have a map $\bigoplus_{y\in \phi^{-1}(x)} E_x\, \longrightarrow\, E_x$ defined by
$\bigoplus_{y\in \phi^{-1}(x)} v_y \, \longmapsto \, \sum_{y\in \phi^{-1}(x)} v_y$, where
$v_y\,\in\, E_x$. This pointwise construction produces the homomorphism $H$ in \eqref{e25}.
On the other hand, we have the homomorphism
\begin{equation}\label{e25b}
G \ :\ E \longrightarrow\ \phi_* \phi^* E
\end{equation}
that sends any $v\, \in\, E_x$ to $\frac{1}{d}\sum_{y\in \phi^{-1}(x)} v$, where $d\,=\, \text{degree}(f)$.
Clearly, we have $H\circ G \,=\, {\rm Id}_E$.

It is straightforward to check that the homomorphism $H$ in \eqref{e25} takes the connection
$\phi_*\phi^*\nabla^E$ on $\phi_* \phi^* E$ to the connection $\nabla^E$ on $E$. Similarly, the
homomorphism $G$ in \eqref{e25b} takes the connection $\nabla^E$ on $E$ to the connection
$\phi_*\phi^*\nabla^E$ on $\phi_* \phi^* E$.

Let $\widehat{\underline E}$ be the locally constant sheaf on $X_0$ given by the locally defined flat sections of
$\phi_*\phi^* E$ for the flat connection $\phi_*\phi^*\nabla^E$. Since $H$ in \eqref{e25} intertwines the
connections $\phi_*\phi^*\nabla^E$ and $\nabla^E$ on $\phi_*\phi^* E$ and $E$ respectively, it produces a
surjective homomorphism
$$
\widehat{\underline E} \ \longrightarrow\ \underline{E}.
$$
As $G$ in \eqref{e25b} also intertwines the
connections $\phi_*\phi^*\nabla^E$ and $\nabla^E$ on $\phi_*\phi^* E$ and $E$ respectively,
it follows that $\underline{E}$ is canonically a direct summand of $\widehat{\underline E}$. In view of the
fact that $\underline{E}$ is canonically a direct summand of $\widehat{\underline E}$, we have a natural
projection
\begin{equation}\label{e27}
\text{End}(\widehat{\underline E})\ \longrightarrow\ \text{End}(\underline{E}).
\end{equation}
More explicitly, if $\widehat{\underline E}\,=\, \underline{E}\oplus {\mathcal A}$, then
$$
\text{End}(\widehat{\underline E})\ =\ \text{End}(\underline{E})\oplus \text{End}({\mathcal A})
\oplus (\underline{E}^*\otimes {\mathcal A})\oplus (\underline{E}\otimes {\mathcal A}^*).
$$
This decomposition produces a projection $\text{End}(\widehat{\underline E})\, \longrightarrow\,
\text{End}(\underline{E})$ as in \eqref{e27}.

Since $H^1_c(Y_0,\, \phi^*\text{End}(\underline{E}))\ =\ H^1_c(X_0,\, \text{End}(\widehat{\underline E}))$,
the homomorphism in \eqref{e27} produces a surjective homomorphism
\begin{equation}\label{e28}
{\mathcal H}\ :\ H^1_c(Y_0,\, \phi^*\text{End}(\underline{E}))\ =\
H^1_c(X_0,\, \text{End}(\widehat{\underline E}))\ \longrightarrow\
H^1_c(X_0,\, \text{End}(\underline{E})).
\end{equation}

The dual of the differential $(d\Phi)^*$ of $\Phi$ at the point $\Phi((E,\, \nabla^E))$ is a homomorphism
$$
T^*_{\Phi((E,\ \nabla^E))} {\mathcal C}_n(Y_0)\ =\
H^1_c(Y_0,\, \phi^*\text{End}(\underline{E})) \ \longrightarrow\
T^*_{(E,\ \nabla^E)} {\mathcal C}_n(X_0)\ =\ H^1_c(X_0,\, \text{End}(\underline{E})).
$$
This homomorphism coincides with the homomorphism ${\mathcal H}$ in \eqref{e28}.

Let ${\mathbb P}\, :\, T^* {\mathcal C}_n(X_0)\, \longrightarrow\, T{\mathcal C}_n(X_0)$ and
${\mathbb Q}\, :\, T^* {\mathcal C}_n(Y_0)\, \longrightarrow\, T{\mathcal C}_n(Y_0)$ be the
holomorphic Poisson structures on ${\mathcal C}_n(X_0)$ and ${\mathcal C}_n(Y_0)$ respectively. From the
above descriptions of $d\Phi$ and $(d\Phi)^*$ it follows that the following diagram is commutative:
$$
\begin{matrix}
T^* {\mathcal C}_n(X_0) & \xrightarrow{\,\,\, {\mathbb P}\,\,\,} & T {\mathcal C}_n(X_0)\\
\, \,\,\,\,\, \Big\uparrow (d\Phi)^* && \,\,\,\, \Big\downarrow d\Phi\\
\Phi^*T^*{\mathcal C}_n(Y_0) & \xrightarrow{\,\,\, \Phi^*{\mathbb Q}\,\,\,} & \Phi^*T{\mathcal C}_n(Y_0)
\end{matrix}
$$
In other words, the map $\Phi$ is compatible with the holomorphic Poisson structures on ${\mathcal C}_n(Y_0)$
and ${\mathcal C}_{n}(X_0)$.
\end{proof}

\begin{proposition}
The symplectic form on $($symplectic leaves of $)$ $\mathcal{M}_r(Y)$ is ${\rm degree}(f)$--times
the pullback of the symplectic form on $\mathcal{M}_r(X)$ under $f\,:\, Y \,\longrightarrow\, X$.
\end{proposition}

\begin{proof}
Let $\alpha_1,\, \alpha_2 \,\in\, H^1(X, \text{End}(\underline{F})).$ These are elements of the tangent space
of $\mathcal{M}_r(X).$ Then the symplectic form on $\mathcal{M}_r(X)$ is
\begin{equation}
\omega (\alpha_1,\, \alpha_2)\ =\ \int_X \text{Trace}(\alpha_1 \wedge \alpha_2).
\end{equation}

Let $\beta_j \,\in\, H^1 (Y, \text{End}(f^*\underline{F}))$ be the pullback
of $\alpha_j$. We then have
$$ \omega (\beta_1,\, \beta_2)\ =\ \int_Y \text{Trace}(\beta_1 \wedge \beta_2)\ =\
\text{degree}(f)\cdot \omega (\alpha_1,\, \alpha_2)$$
since $f^* \alpha_j \,=\, \beta_j.$
\end{proof}

\section{Composition of Maps}

First set $n\,=\, r$ in \eqref{e23}. Let
$$
\Phi_r\ :\ {\mathcal C}_{r}(X_0)\ \longrightarrow\ {\mathcal C}_r(Y_0)
$$
be the map in \eqref{e23}. Consider $\Psi$ in \eqref{e14}. We will describe the composition
of maps $\Psi\circ\Phi_r$.

Let $L_0\, :=\, Y_0\times {\mathbb C}\, \longrightarrow\, Y_0$ be the trivial complex line bundle
equipped with the trivial connection $d^0$ given by the de Rham differential on $C^\infty$
function on $Y_0$. Let
\begin{equation}\label{e29}
({\mathcal E}_0, \, \nabla^0)\ :=\ (\phi_*L_0,\, \phi_*d^0)
\end{equation}
be the flat vector bundle on $X_0$ obtained by taking the direct image of $(L_0,\, d^0)$ (see
\eqref{e31}).

\begin{lemma}\label{lem1}
Take any $\alpha\, :=\, (E,\, \nabla)\, \in\, {\mathcal C}_{r}(X_0)$, where ${\mathcal C}_{r}(X_0)$ is
as in \eqref{e21} (recall that $n\,=\,r$). Then the flat vector bundle
$\Psi\circ\Phi_r (\alpha)$ on $X_0$ is isomorphic to the flat vector bundle
$$(E\otimes {\mathcal E}_0,\ \nabla\otimes {\rm Id}_{{\mathcal E}_0}+{\rm Id}_E\otimes \nabla^0),$$
where $({\mathcal E}_0, \, \nabla^0)$ is constructed in \eqref{e29}.
\end{lemma}

\begin{proof}
Let $\Gamma_0$ be a finitely presented group and $\Delta_0\, \subset\, \Gamma_0$ a subgroup of
finite index. Let $A_0$ be the trivial complex $\Delta_0$--module of dimension one. Let ${\mathbb L}$
be the corresponding $\Gamma_0$--module (see the construction of $W$ in \eqref{e32}).

Now let $V$ be a finite dimensional complex $\Gamma_0$--module. Since $\Delta_0$ is a subgroup of
$\Gamma_0$, restricting the action of $\Gamma_0$ on $V$ to the subgroup $\Delta_0$ we get a
$\Delta_0$--module; this $\Delta_0$--module is denoted by $\mathbb V$. Note that as vector spaces,
$\mathbb V$ coincides with $V$. Now let $\widehat{\mathbb V}$ be the $\Gamma_0$--module corresponding
to the $\Delta_0$--module $\mathbb V$ (see the construction in \eqref{e32}). Then the
$\Gamma_0$--module $\widehat{\mathbb V}$ is isomorphic to the $\Gamma_0$--module $V\otimes {\mathbb L}$.

Now set $\Gamma_0\,=\, \Gamma\,:=\, \pi_1(X_0,\, \phi(y_0))$ and
$\Delta_0\,=\, \Delta\, :=\, \phi_*(\pi_1(Y_0,\, y_0))$. Also set $V$ to be the $\pi_1(X_0,\,
\phi(y_0))$--module given by $E_{\phi(y_0)}$ equipped with the monodromy action of
$\pi_1(X_0,\, \phi(y_0))$ for the flat connection $\nabla$ on $E$. Now
the lemma follows immediately from the above fact. 
\end{proof}

Next set $n\,=\, r\cdot \text{degree}(f)\,=\, rd$ in \eqref{e23}. Let
$$
\Phi_{rd}\ :\ {\mathcal C}_{rd}(X_0)\ \longrightarrow\ {\mathcal C}_{rd}(Y_0)
$$
be the map in \eqref{e23}. Consider $\Psi$ in \eqref{e14}. We will describe the composition
of maps $\Phi_{rd}\circ \Psi$.

\begin{lemma}\label{lem2}
Take any $\beta\, :=\, (E,\, \nabla)\, \in\, {\mathcal C}_{r}(Y_0)$, where ${\mathcal C}_{r}(Y_0)$ is
as in \eqref{e22}. Then the flat vector bundle
$\Phi_{rd}\circ\Psi (\beta)$ on $X_0$ is isomorphic to the flat vector bundle
$$(E^{\oplus d},\ \nabla^{\oplus d}).$$
\end{lemma}

\begin{proof}
Let $\Gamma_0$ be a finitely presented group and $\Delta_0\, \subset\, \Gamma_0$ a subgroup of
finite index. Take a finite dimensional complex $\Delta_0$--module $V$.
Let $\mathcal V$ be the $\Gamma_0$--module corresponding
to the $\Delta_0$--module $V$ (see the construction in \eqref{e32}).
Since $\Delta_0$ is a subgroup of
$\Gamma_0$, restricting the action of $\Gamma_0$ on $\mathcal V$ to the subgroup $\Delta_0$ we get a
$\Delta_0$--module; this $\Delta_0$--module is denoted by $\widehat{\mathcal V}$. This
$\Delta_0$--module $\widehat{\mathcal V}$ is isomorphic to the $\Delta_0$--module
$V^{\oplus d}$. 

Set $\Gamma_0\,=\, \Gamma\,:=\, \pi_1(X_0,\, \phi(y_0))$ and
$\Delta_0\,=\, \Delta\, :=\, \phi_*(\pi_1(Y_0,\, y_0))$. Also set $V$ to be the $\pi_1(Y_0,\,
\phi(y_0))$--module given by $E_{y_0}$ equipped with the monodromy action of
$\pi_1(Y_0,\, y_0)$ for the flat connection $\nabla$ on $E$. 
Now the lemma follows from the above fact.
\end{proof}

\section{Isomonodromy}

The preceding sections have used quite extensively flat connections associated to representations. If we 
endow $Y_0\,\subset\, Y$ with the structure of a Riemann surface, we can take these connections to be 
holomorphic. Let $\mathcal{N}_r(Y_0)$ represent the moduli space of holomorphic connections on the surface $Y_0$, 
with logarithmic singularities at $Y-Y_0$. There is then a natural map $$\mu\ :\ \mathcal{N}_r (Y_0)\
\longrightarrow\
{\mathcal C}_r(Y_0)$$ associating to each connection its monodromy. Both spaces have complex Poisson 
structures, the one on $\mathcal{N}_r (Y_0)$ defined using the particular complex structure, while the one on 
${\mathcal C}_r(Y_0)$ does not. This map is a local biholomorphism, and it is shown in \cite{BJ, BJ2} that 
the map is also a Poisson isomorphism.

Now if we have a deformation $Y_0(t)$, we have a corresponding isomonodromic deformations of connections in 
the family $\mathcal{N}_r(Y_0(t))$, given simply by considering the maps $\mu(t)$ of $\mathcal{N}_r (Y_0)(t)$
into ${\mathcal C}_r(Y_0)(t)\, =\, {\mathcal C}_r(Y_0)(0)$, and fixing the image of the map. As these maps
are Poisson, the composition $\mu(t)\mu(0)^{-1}$, the isomonodromy deformation map, is also.

Our theorem allows us to extend this to a relative context. Indeed, a family $Y_0(t)\,\longrightarrow\,
X_0(t)$ gives a commuting diagram 
$$\begin{matrix} \mathcal{N}_r(Y_0(t))& \longrightarrow &{\mathcal C}_r(Y_0)\\
\Big\downarrow && \Big\downarrow\\ 
\mathcal{N}_{rd}(X_0(t))& \longrightarrow &{\mathcal C}_{rd}(X_0)\end{matrix}$$
all of whose maps are Poisson. It follows that the diagram of isomonodromy maps, 
$$\begin{matrix} \mathcal{N}_r(Y_0(t))& \longrightarrow &\mathcal{N}_r(Y_0(0))\\
\Big\downarrow && \Big\downarrow\\ 
\mathcal{N}_{rd}(X_0(t))& \longrightarrow &\mathcal{N}_{rd}(X_0(0))\end{matrix}$$
induced by a variation $Y_0(t)\,\longrightarrow\, X_0(t)$ are also Poisson, in a way compatible
with the diagram.

\end{document}